\documentclass[11pt]{amsart}

\usepackage[utf8]{inputenc}
\usepackage{amsmath, amssymb, amsthm}
\usepackage{geometry}
\usepackage{etoolbox}
\newcommand{\Nat}{\mathbf{N}}
\newcommand{\FS}{\mathrm{FS}}

\newtheorem*{theorem*}{Theorem}
\newtheorem{theorem}{Theorem}[section]

\newtheorem{proposition}{Proposition}[section]

\makeatletter
\AtBeginEnvironment{proof}{\setcounter{claim}{0}}
\makeatother

\theoremstyle{definition}

\title{An equivalent form of Hindman's Theorem}
\author{Lorenzo Carlucci}
\address{Dipartimento di Matematica\\ Sapienza University of Rome\\
Piazzale Aldo Moro 5, 00185 Roma, Italy}
\email{lorenzo.carlucci@uniroma1.it}

\author{Andrea D'Amico}
\address{
Dipartimento di Informatica\\ Sapienza University of Rome\\
 Viale Regina Elena, 295, 00161 Roma, Italy}
\email{damico.2046502@studenti.uniroma1.it}

\date{}

\begin{document}

\begin{abstract}
We prove the equivalence of Hindman's finite sums theorem with a 
natural extension of Ramsey's theorem. 
\end{abstract}

\maketitle

\section{Introduction}
The following theorem was proved by Neil Hindman in \cite{Hin:74}: 

\begin{theorem*}[Hindman's finite sums theorem]
Every finite coloring $c$ of the positive integers admits an infinite increasing
sequence $X = (x_j)_{j \in\Nat}$ of positive integers such that all finite non-empty sums of distinct elements from $X$ have the same color under $c$.
\end{theorem*}


For any set $X\subseteq\Nat$ we denote by $\FS(X)$ the set of all finite non-empty sums of distinct elements from $X$; we denote by $\mathcal{E}(X)$ the set of all finite non-empty subsets of $X$ with even cardinality. Let $Y\in \mathcal{E}(\Nat)$ and let $\{y_1,y_2,\dots,y_{2n}\}$ be the enumeration of $Y$ in increasing order. We define its alternating sum of consecutive differences as follows:
 $$A(Y):=\sum_{i=1}^{n}(y_{2i}-y_{2i-1}).$$

We call a coloring $f: \mathcal{E}(\Nat) \to k$ {\em invariant} if it satisfies the following property: for all $Y, Z \in \mathcal{E}(\Nat)$ such that $A(Y) = A(Z)$, the equality $f(Y) = f(Z)$ holds. 


Our main goal in this note is to prove the equivalence of Hindman's theorem with the following theorem.

\begin{theorem*}[Theorem A]
For every invariant $k$-coloring $f:\mathcal{E}(\Nat)\rightarrow k$, there exists an infinite increasing sequence $S$of positive integers such that $f$ is constant on all non-empty subsets of $S$ of even cardinality.
\end{theorem*}

Note that the conclusion of Theorem A fails for arbitrary $f:\mathcal{E}(\Nat)\rightarrow k$. A simple counterexample is given by coloring $Y \in \mathcal{E}(\Nat)$ blue if $|Y| = 2 \pmod 3$ and red otherwise. Any infinite set contains subsets of cardinality $2$ and of cardinality $4$ that get different colors. The coloring is obviously not invariant. We prove the following theorem.

\begin{theorem}\label{thm:main}
Hindman's theorem is equivalent to Theorem A.
\end{theorem}

The result generalizes the equivalence proved in \cite{Bre:26} between the restriction of Hindman's theorem to sums of consecutive elements of the solution sequence (introduced by the first author in \cite{Car:18:weak} under the name of Adjacent Hindman's Theorem) and the restriction of Theorem A to sets of cardinality $2$, called the $\mathbf{Z}$-invariant Ramsey's Theorem in \cite{Bre:26}.

We split the equivalence proof in two parts. 

\section{A proof of Theorem A from Hindman's theorem}

\begin{proposition}
Hindman's theorem implies Theorem A.
\end{proposition}

\begin{proof}

    Let $f: \mathcal{E}(\Nat) \to k$ be an invariant coloring. Define $c: \Nat \to k$ as follows, for $x >0$:
    $$c(x) := f(\{0, x\}),$$
 and set $c(0):=0$. By Hindman's Theorem there exists an infinite sequence $X = (x_j)_{j \in \Nat}$ of positive integers and a fixed color $c^* < k$ such that all finite sums of elements from $X$ are colored $c^*$. 


Define the prefix sum sequence $S = (s_j)_{j \in \Nat}$ by setting:
    $$s_j := \sum_{w=1}^j x_w.$$
We want to show that all elements of $\mathcal{E}(S)$ have color $c^*$ under $f$.
  
Notice that, since $f$ is invariant, for any even cardinality non-empty set $Y$ of natural numbers (be it a subset of $S$ or not), we have the following equality: 
$$A(\{0, A(Y)\})=A(Y)-0=A(Y),$$ 
which yields $f(Y)=f(\{0, A(Y)\})$. Since $f(\{0, A(Y)\})=c(A(Y))$ by definition of $f$, it is sufficient for us to show that 
$c(A(Y))=c^*$ for every $Y \in \mathcal{E}(S)$. For this it is sufficient to show that, for such a $Y$, $A(Y) \in \FS(X)$. 

Let $Y = \{s_{j_1}, s_{j_2}, \dots, s_{j_{2n}}\}_<$ be any non-empty subset of $S$ with even cardinality. Consider the alternating sum of consecutive differences of elements of $Y$:
$$A(Y) = \sum_{i=1}^n (s_{j_{2i}} - s_{j_{2i-1}}).$$
    Recalling the prefix sum definition $s_j = \sum_{w=1}^j x_w$, we have that:
    $$s_{j_{2i}} - s_{j_{2i-1}} = \left(\sum_{w=1}^{j_{2i}} x_w\right) - \left(\sum_{w=1}^{j_{2i-1}} x_w\right) = \sum_{w=j_{2i-1}+1}^{j_{2i}} x_w.$$
Setting $F = \bigcup_{i=1}^n F_i$, with $F_i=[j_{2i-1}+1, j_{2i}]$ for $i$ such that $1\leq i \leq n$, we can write $A(Y)$ as a sum of finitely many distinct elements from $X$ as follows:
    $$A(Y)=\sum_{i=1}^{n}(s_{j_{2i}}-s_{j_{2i-1}})=\sum_{i=1}^{n}\sum_{w=j_{2i-1}+1}^{j_{2i}}x_w=\sum_{w \in F} x_w.$$
By hypothesis on $X$ we conclude that
    $$c(A(Y))=c\left(\sum_{w \in F} x_w\right) = c^*.$$
Therefore, for all $Y \in \mathcal{E}(S)$, the following holds:
    $$f(Y)=f(\{0, A(Y)\})=c(A(Y)) = c^*.$$
This proves Theorem A.
\end{proof}

\section{A proof of Hindman's theorem from Theorem A}

\begin{proposition}\label{prop:AtoHT}
Theorem A implies Hindman's theorem.
\end{proposition}

\begin{proof}
Let $c: \Nat \to k$. Define a coloring $f: \mathcal{E}(\Nat) \to k$ by setting, for $Y \in \mathcal{E}(\Nat)$,
    $$f(Y) := c(A(Y)).$$
The coloring $f$ is obviously invariant. Therefore, by Theorem A, there exists an infinite increasing sequence $S = (s_j)_{j \in \Nat}$ and a color $c^*<k$ such that all non-empty subsets of $S$ with even cardinality have color $c^*$ under $f$. 

Define $X = (x_j)_{j \in \Nat}$ by setting:
    $$x_j := s_{j+1} - s_j.$$
Without loss of generality (by possibly skipping some elements of the original set $S$) we can assume that $X$ is an infinite increasing sequence.

Let $F \subset \Nat$ be an arbitrary non-empty finite set of indices. We can uniquely partition $F$ into a finite sequence of $n\geq 1$ maximal contiguous intervals
    $$F = F_1 \cup F_2 \cup \dots \cup F_n,$$
    where $F_v = [a_v, b_v]$ and the gaps between intervals satisfy $a_{v+1} > b_v + 1$. When we evaluate the sum of elements from $X$ with indices in $F$, the result is the sum telescoping with respect to the $F_v$ blocks:
    $$\sum_{w \in F} x_w = \sum_{v=1}^n \sum_{w=a_v}^{b_v} x_w = \sum_{v=1}^{n}\left((s_{a_v+1}-s_{a_v})+(s_{a_v+2}-s_{a_v+1})+\dots+(s_{b_v+1}-s_{b_v})\right) =\sum_{v=1}^n (s_{b_v+1}-s_{a_v}).$$
    Since $s_{j+1} > s_j$ and $a_{v+1} > b_v + 1$, we are guaranteed that $s_{a_{v+1}} > s_{b_v+1}$. This means that the indices $a_v,b_v+1$, for $v$ such that $1\leq v \leq n$, appearing in the sum $\sum_{w \in F}x_w$ form a non empty subset of $S$ of even cardinality:
    $$Y = \{s_{a_1}, s_{b_1+1}, s_{a_2}, s_{b_2+1}, \dots, s_{a_n}, s_{b_n+1}\}_<.$$
    It is now easy to observe that $A(Y)=\sum_{w \in F}x_w$. Since $Y \in \mathcal{E}(S)$, we have $f(Y) = c^*$ by hypothesis on $S$. By definition of $f$ we have the following:
$$c\left(\sum_{w \in F} x_w\right) = c(A(Y)) = f(Y) = c^*.$$
    Thus $c$ is constant on $\FS(S)$ as required by Hindman's Theorem.

\end{proof}

\section{Conclusion}

The proof of Proposition \ref{prop:AtoHT} yields that Theorem A restricted to sets of cardinality equal to $2$ and $4$ implies 
the restriction of Hindman's Theorem to sums of at most two elements. Hindman, Leader and Strauss \cite{HLS:16} asked whether
the latter restriction admits a proof that does not also prove the full Hindman's Theorem. Such a proof of the 
restriction of Theorem A to cardinalities $2$ and $4$ would answer the question in the positive. 

Hindman's Theorem is also of interest in reverse mathematics and computability theory, 
where one is interested in the logical and algorithmic complexity of solutions to the theorem as a function of the complexity of the instance 
(see \cite[Section 9.9]{Dza-Mum:22}). The equivalence of Hindman's theorem with a natural Ramsey-type principle established in the present note can offer new insights in this respect, since standard Ramsey-type theorems have been thoroughly investigated (see \cite[Chapter 8]{Dza-Mum:22}). We observe that the proofs establish uniform computable reductions (more precisely strong Weihrauch reductions, see \cite{Dza-Mum:22}) between the two theorems involved.

\end{document}